\documentclass[12pt,reqno]{amsart}
\usepackage{amssymb,latexsym,amsmath}
\usepackage{setspace}
\newtheorem{theorem}{Theorem}
\newtheorem{lemma}{Lemma}

\theoremstyle{definition}
\newtheorem{definition}{Definition}

\begin{document}

\title{A prime power equation}
\author{Timothy Redmond, Charles Ryavec}
\date{September 15, 2022}
\maketitle

\begin{abstract}
A real valued function, $G$, is provided whose Fourier transform, $\hat G$, is an entire function that satisfies, $E(s)\zeta(s) = \hat G(\frac{s -\frac{1}{2}}{i})$.  Then $\hat G(\gamma) = 0$ for all nonreal zeros, $\rho = \frac{1}{2} + i \gamma$, of $\zeta(s)$. Combined with Guinand's explicit formula we obtain a prime power equation free of zeta zeros. Using infinitely many translates of G, an infinite system of equations, indexed on the natural numbers, is obained. The solution vector of this system is the vector of values of von Mangoldt's function, $\Lambda(n), n = 1, 2 \cdots$. The entries of the matrix are special values of the fourth power of the Jacobi theta function, $\theta_2(\tau)$.  
\end{abstract}
\section {Explicit Formulas}
The explicit formulas of prime number theory, as discovered by Riemann, and generalized subsequently, provide a balance between two unknown sequences within the framework of a Fourier transform pair, $G$ and $\hat G$. The explicit formula considered in this note is that of Guinand:
\begin{eqnarray*}
\sum_{\rho = \frac{1}{2} + i\gamma}\hat G(\gamma) &=& \hat G(-\frac{i}{2}) + \hat G(\frac{i}{2}) \\
&=& -\sum_{p^m}\frac{logp}{p^{\frac{m}{2}}}\Big(G(m \ logp) + G(-m \ logp)\Big) \\
&- & log \pi \ G(0) \\
&+& \frac{1}{2\pi}\int_{-\infty}^{\infty}Re\Big(\frac{\dot \Gamma}{\Gamma}(\frac{1}{4} + \frac{i t}{2})  \Big)\hat G(t) dt,   
\end{eqnarray*} 
where the series on the left is the sum over nonreal zeros of the Riemann zeta function, and in the series on the right, a sum over rational prime powers. 

Various strategies have been brought to bear on this relationship in order to gain results on zeros. The consideration of the left side,
\begin{equation*}
\sum_{\rho = \frac{1}{2} + i\gamma}\hat G(\gamma),
\end{equation*}
over particular classes of G may be found in [1], [2], articles which provide references and discussion. The first step here is to eliminate the sum,
\begin{equation*}
\sum_{\rho = \frac{1}{2} + i\gamma}\hat G(\gamma),
\end{equation*}
altogether. To this end a function, $G$, is produced whose Fourier transform, $\hat G$, is entire and factors as,
\begin{eqnarray*}
E(s)\zeta(s) = \hat G(\frac{s - \frac{1}{2}}{i}),
\end{eqnarray*} 
so that 
\begin{eqnarray*}
\hat G(\gamma) = 0
\end{eqnarray*} 
for all nonreal zeros, $\rho = \frac{1}{2} +i\gamma$, of $\zeta(s)$. (The $\gamma$ are not assumed to be real.)  In this case an explicit formula,
\begin{eqnarray*}
 \sum_{p^m}\frac{logp}{p^{\frac{m}{2}}}\Big(G(m \ logp) + G(-m \ logp)\Big) &=& V(G) \\
 V(G) &=& \hat G(-\frac{i}{2}) + \hat G(\frac{i}{2}) \notag \\ 
&- & log \pi \ G(0) \notag \\
&+& \frac{1}{4\pi}\int_{-\infty}^{\infty}\Big(\frac{\dot \Gamma}{\Gamma}(\frac{1}{4} + \frac{i t}{2}) + \frac{\dot \Gamma}{\Gamma}(\frac{1}{4} - \frac{it}{2}) \Big)\hat G(t) dt, \notag  
\end{eqnarray*} 
exists, and a sum on prime powers is thereby expressed in terms of quantities with no apparent connection with the zeros of the Riemann zeta function. It is readily checked that the equation,
\begin{eqnarray*}
\sum_{n=1}^{\infty}\frac{\Lambda(n)}{\sqrt{n}}\Big(G(log(x_n) + G(-log(x_n)\Big) = V(G),
\end{eqnarray*}  
does not identify the prime powers; i. e., $x_n = n$. To pinpoint the prime powers as a unique solution, it would be expected that infinitely many equations would be needed. Given the one pair, $G(v)$ and $\hat G(t)$, the obvious candidates for further $G$ are convolutions, the translations, $G(v + c)$, being the  simplest, with transforms $e^{-i c t}\hat G (t)$.  In a coming section, we use the particular translates, $[ c = log m : m = 1, 2, 3, \cdots]$, obtaining an infinite system of equations,
\begin{equation*}
\sum_{n=1}^{\infty} \Lambda(n) f(m, n) = V(m),
\end{equation*}
indexed by the natural numbers, $[m = 1, 2, \cdots]$. The entries of the column vector,
\begin{equation*}
V = [V(m)]_{m =1}^{\infty}, 
\end{equation*}
are the values of the $m$-th translate, $V(G(v + logm))$. The objective would be a unique solution vector,   
\begin{equation*}
\Lambda = [\Lambda(n)]_{n=1}^{\infty},
\end{equation*}
of
\begin{equation*}
T \Lambda  = V,
\end{equation*}
\begin{equation*}
T = [f (m, n)]_{m, n = 1}^{\infty},
\end{equation*}
with $\Lambda(n)$ being the von Mangoldt function. The entries, $f(m, n)$, (in the $m$-th row and $n$-th column) of $T$ are essentially special values of the modular function,
\begin{equation*}
\lambda(\tau) = \Big(\frac{\theta_2(\tau)}{\theta_3(\tau)}\Big)^4.
\end{equation*}
These are numbers from finite algebraic extensions of the rational field, $Q$. It was not shown that $T$ guaranteed a unique solution.  Alternatively, instead of creating equations from the translates of one $G$, a possible area to look into for other $G$ are Galois extensions, $L:Q$,
with
\begin{equation*}
E_{L}(s)\zeta_L(s) = \hat G_L(\frac{s - \frac{1}{2}}{i}),
\end{equation*}
 since $\zeta(s)$ is a factor of all such $\zeta_L$. We checked that this was possible for 15 imaginary quadratic extensions of $Q$, and found that such $\hat G_L(\frac{s - \frac{1}{2}}{i})$ existed. The new ingredient is a whole new set of unrelated $\hat G_L$. 
\section {A Zeta Function Representation }
We define two multiplicative functions, $a(n)$ and $b(n)$.
\begin{definition}Define 
\begin{eqnarray*}
a(n) &=& \sigma_1(n), \qquad (n, 2) = 1 \\
 &=& 2^m \qquad \qquad n = 2^m, 
\end{eqnarray*}
 \end{definition}
and 
\begin{eqnarray*} 
b[n] &=& \sigma_1(n) \qquad (n, 2) = 1 \\
 &=& 0 \qquad \ \ \ \ \ \ n = 2^m \ \ \ m\ge 1.
\end{eqnarray*}
Note that differentiating the theta function, 
\begin{eqnarray*}
\theta_4(\tau) = \sum_{n=-\infty}^{\infty}(-1)^n q^{n^2}\qquad q= e^{i\pi \tau}, \tau = i y,
\end{eqnarray*}
with respect to y, has coefficients given by,
\begin{eqnarray*}
\frac{\dot\theta_4}{\theta_4}(i y) &=& -2\pi\sum_{n=1}^{\infty}a(n)e^{-\pi n y},
\end{eqnarray*}
while the fourth power of the theta function, $\theta_2(\tau)$, has coefficients,
\begin{eqnarray*}
\theta_2^4(\tau) &=&  \Big(\sum_{n= -\infty}^{\infty}q^{(n+\frac{1}{2})^2}\Big)^4 \\
&=&16\sum_{n=1}^{\infty}b(n)q^n.
\end{eqnarray*}
\begin{lemma}
\begin{eqnarray*}
\frac{s(s+1)}{32 \pi^2\sqrt{2}}(2^{\frac{s}{2}} - 2^{-\frac{s}{2}})(2^{\frac{s-1}{2}} - 2^{-\frac{s-1}{2}})\zeta^{*}(s)\zeta^{*}(s+1) &=& \sum_{j=1}^{\infty}\sum_{d|j}a(d)b(\frac{j}{d})\Big(\frac{j}{d^2}\Big)^{\frac{s}{2}}K(s, 2\pi\sqrt{j}), \notag \\
\zeta^{*}(s) &=& \pi^{-\frac{s}{2}}\Gamma(\frac{s}{2})\zeta(s). \notag
\end{eqnarray*}
\end{lemma}
Here, $K(s, x)$ is the modified Bessel function of the second kind, usually written, $K_s(x)$.
\begin{proof}
As before, let
\begin{equation*}
\theta_4(\tau) = \sum_{k=-\infty}^{\infty}(-1)^k q^{k^2} \qquad q = e^{i \pi \tau}.
\end{equation*}
Then
\begin{equation*}
\theta_4^4(\tau) -1 = -8 \sum_{k= 1}^{\infty}c(k) q^k,
\end{equation*}
with the multiplicative coefficients,
\begin{eqnarray*}
c(k) &=& \sigma_1(k), \qquad (k, 2) = 1 \\
 &=& -3 \qquad \qquad k = 2^m, \ 1 \le m.
\end{eqnarray*}
With $\tau = i y$, and integrating $y$ over $(0, \infty)$, there results,
\begin{eqnarray*}
\int_0^{\infty}y^{2w-1}(\theta_4^4(e^{-\pi y}) -1)dy &=& -8\pi^{-2w}\Gamma(2w) \sum_{k=1}^{\infty}c(k) k^{-2w} \\
&=& -8\pi^{-2w}\Gamma(2w)\sum_{(k,2) =1}\Big(1 - 3\sum_{m=1}^{\infty}2^{-2mw}\Big)c(k)k^{-2w} \\
&=& -8\pi^{-2w}\Gamma(2w)(1 - 2^{2 - 2w})\zeta(2w - 1)(1 - 2^{-(2w-1)})\zeta(2w) \\
&=& -8\pi^{-(s+1)}\Gamma(s+1)(1 - 2^{1 - s})\zeta(s)(1 - 2^{-s})\zeta(s+1) ,
\end{eqnarray*}
where, in the last line, we put $2w = s + 1$. Returning to the integral in the first line, and integrating by parts, we have,
\begin{eqnarray*}
\int_0^{\infty}y^{2w-1}(\theta_4^4(e^{-\pi y}) -1)dy &=&- \frac{4}{2w}\int_0^{\infty}y^{2w}\frac{\dot\theta_4}{\theta_4}(e^{-\pi y})\theta_4^4(e^{-\pi y})dy \\
&=& -\frac{4}{2w}\int_0^{\infty}y^{2w}\frac{\dot\theta_4}{\theta_4}(e^{-\pi y})y^{-2}\theta_2^4(e^{-\frac{\pi}{ y}})dy \\
&=& \frac{4}{s+1}\int_0^{\infty}y^{s - 1}2\pi \sum_{k=1}^{\infty}a(k)e^{-\pi k y}(16)\sum_{l=1}^{\infty}b(l)e^{-\frac{\pi l}{y}}dy \\
&=& \frac{4(2\pi)(16)}{s+1}\int_0^{\infty}y^{s - 1} \sum_{k=1}^{\infty}a(k)e^{-\pi ky}\sum_{l=1}^{\infty}b(l)e^{-\pi \frac{\pi l}{y}}dy \\
&=& \frac{4(32)\pi}{s+1} \sum_{k=1}^{\infty}\sum_{l=1}^{\infty}a(k)b(l)\int_0^{\infty}y^{s - 1}e^{-\pi(ky+\frac{l}{y})}dy \\
&=& \frac{4(32)\pi)}{s+1} \sum_{j=1}^{\infty}\sum_{d|j}a(d)b(\frac{j}{d})\Big(\frac{j}{d^2}\Big)^{\frac{s}{2}}\int_0^{\infty}y^{s - 1}e^{-\pi(ky+\frac{l}{y})}dy \\
&=& \frac{8(32)\pi}{s+1} \sum_{j=1}^{\infty}\sum_{d|j}a(d)b(\frac{j}{d})\Big(\frac{j}{d^2}\Big)^{\frac{s}{2}} K(s, 2\pi \sqrt{j}) .
\end{eqnarray*}
Equating,
\begin{equation*} 
-8\pi^{-(s+1)}\Gamma(s+1)(1 - 2^{1 - s})\zeta(s)(1 - 2^{-s})\zeta(s+1) = \frac{8(32)\pi}{s+1} \sum_{j=1}^{\infty}\sum_{d|j}a(d)b(\frac{j}{d})\Big(\frac{j}{d^2}\Big)^{\frac{s}{2}} K(s, 2\pi \sqrt{j})
\end{equation*}
and simplifying, the lemma follows. 
\end{proof}
\section {$G(v)$ and $\hat G(t)$}
With $K$ the modified Bessel function of the second kind, then, with $s = \frac{1}{2} + i t$, the Mellin transform is,  
\begin{eqnarray*}
K(\frac{1}{2} + i t, x) &=& \frac{1}{2}\int_0^{\infty}u^{i t} u^{-\frac{1}{2}} exp[-\frac{x}{2}(u + \frac{1}{u})]du. \\
 &=& \frac{1}{2}\int_{-\infty}^{\infty}e^{i t v} e^{\frac{v}{2}}  exp[-\frac{x}{2}(e^v + e^{-v}))]dv \\
&=& \frac{1}{2}\int_{-\infty}^{\infty}e^{i t v}   g[v, x] dv \\ 
&=& \frac{1}{2} \hat g(t, x),
\end{eqnarray*} 
and the resulting pair of transforms collected here for future reference, 
\begin{eqnarray*}
g(v, x) &=& e^{\frac{v}{2}}  exp[-\frac{x}{2}(e^v + e^{-v}))] \\
 \hat g(t, x) &=& \int_0^{\infty}e^{i t v}g(v, x) dv \\
 &=& 2K(\frac{1}{2} + i t, x).  
\end{eqnarray*}
Note that if $\rho = \frac{1}{2} + i\gamma$, then
\begin{eqnarray*}
\hat g(\gamma, x) = 2 K(\rho, x).
\end{eqnarray*}
Next, with $j$ a positive integer, and $d$ a positive integer divisor of $j$, let
\begin{eqnarray*}
T &=& \frac{1}{2} Log\Big(\frac{j}{d^2}\Big).
\end{eqnarray*}
Define
\begin{eqnarray*}
G(v, d, j) &=& g(v - T, 2\pi\sqrt{j}) \\
&=& e^{\frac{v - T}{2}}exp[-\pi\sqrt{j}(e^{v- T} + e^{-(v - T)})] \\
&=& \Big(\frac{j}{d^2}\Big)^{-\frac{1}{4}}e^{\frac{v}{2}}exp[-\pi\sqrt{j}\Big(e^v\frac{d}{\sqrt{j}} + e^{-v}\frac{\sqrt{j}}{d}\Big)] \\
&=& \Big(\frac{j}{d^2}\Big)^{-\frac{1}{4}}e^{\frac{v}{2}}exp[-\pi\Big(e^v d + e^{-v}\frac{j}{d}\Big)].
\end{eqnarray*}
Then,
\begin{eqnarray*}
\hat G[t, d, j] & =&  \int_{-\infty}^{\infty}e^{i t v} g[v - T, 2\pi\sqrt{j}] dv \\
&=&e^{i t T}\hat g[t, 2\pi\sqrt{j}] \\
&=& \Big(\frac{j}{d^2}\Big)^{\frac{i t}{2}}\hat g[t, 2\pi\sqrt{j}], \\
\end{eqnarray*}
and therefore,
\begin{eqnarray*}
\hat G[\gamma, d, j] &=&  \Big(\frac{j}{d^2}\Big)^{\frac{i \gamma}{2}}\hat g[\gamma, 2\pi\sqrt{j}]\\
&=& \Big(\frac{j}{d^2}\Big)^{\frac{i \gamma}{2}} 2K(\rho, x).
\end{eqnarray*}
Finally, define,
\begin{definition} 
\begin{eqnarray*}
G(v) &=& \sum_{j=1}^{\infty}\sum_{d | j}a(d)b(\frac{j}{d}) \Big(\frac{j}{d^2}\Big)^{\frac{1}{4}}G[v, d, j]. 
 \end{eqnarray*}
\end{definition}
Then,
\begin{eqnarray*}
 \hat G(t) &=& \sum_{j=1}^{\infty}\sum_{d | j}a(d)b(\frac{j}{d}) \Big(\frac{j}{d^2}\Big)^{\frac{1}{4}}  \hat G[t, d, j] \\
 &=& \sum_{j=1}^{\infty}\sum_{d | j}a(d)b(\frac{j}{d}) \Big(\frac{j}{d^2}\Big)^{\frac{1}{4}}\Big(\frac{j}{d^2}\Big)^{\frac{i t}{2}}\hat g[t, 2\pi\sqrt{j}] \\
 &=&\sum_{j=1}^{\infty}\sum_{d | j}a(d)b(\frac{j}{d}) \Big(\frac{j}{d^2}\Big)^{\frac{\frac{1}{2} + i t}{2}}\hat g[t, 2\pi\sqrt{j}] \\ 
  &=&\sum_{j=1}^{\infty}\sum_{d | j}a(d)b(\frac{j}{d}) \Big(\frac{j}{d^2}\Big)^{\frac{\frac{1}{2} + i t}{2}}2K(\frac{1}{2} + i t, 2\pi\sqrt{j}] 
 \end{eqnarray*}
\begin{theorem}
With G defined in Definition 2,
\begin{eqnarray*} 
\sum_{\gamma}\hat G(\gamma) &=& \sum_{\gamma}\sum_{j=1}^{\infty}\sum_{d|j}a(d)b(\frac{j}{d})\Big(\frac{j}{d^2}\Big)^{\frac{1}{4} + i\frac{\gamma}{2}}2 K(\frac{1}{2} + i\gamma, 2\pi\sqrt{j}) \\
&=& 0.
\end{eqnarray*}
\end{theorem}
\begin{proof}
\begin{eqnarray*}
\hat G(\gamma) &=& \sum_{j=1}^{\infty}\sum_{d|j}a(d)b(\frac{j}{d})\Big(\frac{j}{d^2}\Big)^{\frac{1}{4} + i\frac{\gamma}{2}}2 K(\frac{1}{2} + i\gamma, 2\pi\sqrt{j}) \\
&=& \sum_{j=1}^{\infty}\sum_{d|j}a(d)b(\frac{j}{d})\Big(\frac{j}{d^2}\Big)^{\frac{\rho}{2}} 2 K(\rho, 2\pi\sqrt{j}) \\
&=& 0 
\end{eqnarray*}
by Lemma 1.
\end{proof}
\section {The Prime Power Equation}
\begin{definition} 
\begin{eqnarray*}
\alpha(u) &=& \sum_{k=1}^{\infty} a(k)e^{-\pi k u} \\
\beta(u)&=& \sum_{l=1}^{\infty} b(l)e^{-\pi l u}\qquad u > 0.
 \end{eqnarray*}
 \end{definition}
With $G$ and $\hat G$ defined in section 3, the expressions in the prime power equation,
\begin{eqnarray*}
 \sum_{p^m}\frac{logp}{p^{\frac{m}{2}}}\Big(G(m \ logp) + G(-m \ logp)\Big) &=& \hat G(-\frac{i}{2}) + \hat G(\frac{i}{2}) \\ 
&- & log \pi \ G(0) \\
&+& \frac{1}{4\pi}\int_{-\infty}^{\infty}\Big(\frac{\dot \Gamma}{\Gamma}(\frac{1}{4} + \frac{i t}{2}) + \frac{\dot \Gamma}{\Gamma}(\frac{1}{4} - \frac{it}{2}) \Big)\hat G(t) dt,   
\end{eqnarray*} 
may be recast in terms of $\alpha$ and $\beta$, which is done in the next four lemmas. Continuing,
\begin{lemma} 
\begin{eqnarray*}
\sum_{p^m}\frac{logp}{p^{\frac{m}{2}}}\Big(G(m \ logp) + G(-m \ logp)\Big) = \sum_{n=1}^{\infty}\Lambda (n) \  \alpha( n)\beta(\frac{1}{ n}) +  \sum_{n=1}^{\infty}\frac{\Lambda(n)}{n} \alpha(\frac{1}{n})\beta(n).
 \end{eqnarray*}
\end{lemma}
\begin{proof}
Recalling,
\begin{equation}
g[v, 2\pi\sqrt{j}] = e^{\frac{v}{2}}  exp[-\pi\sqrt{j}(e^v + e^{-v}))], \notag
\end{equation}
then,
\begin{eqnarray*}
G(m logp, d, j) &=& g(m \ logp - T, 2\pi\sqrt{j}) \\ 
&=& e^{\frac{m logp - T}{2}}  exp[-\pi\sqrt{j}(e^{m logp - T} + e^{-(m logp - T)})] \\
&=& p^{\frac{m}{2}}\Big(\frac{d^2}{j}\Big)^{\frac{1}{4}}exp[-\pi\sqrt{j}\Big(p^m \frac{d}{\sqrt{j}} + p^{-m}\frac{\sqrt{j}}{d}\Big)] \\
&=& p^{\frac{m}{2}}\Big(\frac{d^2}{j}\Big)^{\frac{1}{4}}exp[-\pi \Big(p^m d +  \ \frac{j}{p^m d}\Big)],
\end{eqnarray*}
so that
\begin{eqnarray*}
G(m \ logp) &=& p^{\frac{m}{2}}\sum_{j=1}^{\infty}\sum_{d|j} a[d]b[\frac{j}{d}]\Big(\frac{j}{d^2}\Big)^{\frac{1}{4}}\Big(\frac{d^2}{j}\Big)^{\frac{1}{4}}exp[-\pi \Big(p^m d +  \ \frac{j}{p^m d}\Big)] \\ 
&=&  p^{\frac{m}{2}}\sum_{j=1}^{\infty}\sum_{d|j} a[d]b[\frac{j}{d}]exp[-\pi \Big(p^m d +  \ \frac{j}{p^m d}\Big)]  \\
&=& p^{\frac{m}{2}}\sum_{k=1}^{\infty}a[k]exp[-\pi p^m k]\sum_{l=1}^{\infty}b[l]  exp[-\pi p^{-m} l] \\
&=& p^{\frac{m}{2}}\alpha(p^m)\beta(p^{-m}), \\
\sum_{p^m}\frac{logp}{p^{\frac{m}{2}}}G(m \ logp) &=& \sum_{p^m}logp \  \alpha(p^m)\beta(p^{-m}),
\end{eqnarray*}
and
\begin{eqnarray*}
G(-m \ logp , d, j) &=&  g(-m \ logp - T, 2\pi\sqrt{j}) \\ 
&=& e^{\frac{-m logp - T}{2}}  exp[-\pi\sqrt{j}(e^{-m logp - T} + e^{-(-m logp - T)})] \\
&=& p^{-\frac{m}{2}}\Big(\frac{d^2}{j}\Big)^{\frac{1}{4}}exp[-\pi\sqrt{j}\Big(p^{-m} \frac{d}{\sqrt{j}} + p^{m}\frac{\sqrt{j}}{d}\Big)] \\
&=& p^{-\frac{m}{2}}\Big(\frac{d^2}{j}\Big)^{\frac{1}{4}}exp[-\pi \Big(p^{-m} d +  \ p^m\frac{j}{d}\Big)],
\end{eqnarray*}
\begin{eqnarray*}
\sum_{p^m} \frac{log p}{p^{\frac{m}{2}}}G(-m log p) &=& \sum_{p^m}\frac{log p}{{p^m}}\sum_{k=1}^{\infty}a[k]exp[-\pi p^{-m} k]\sum_{l=1}^{\infty}b[l] exp[-\pi p^{m} l]\\
&=& \sum_{p^m}\frac{log p}{{p^m}}\alpha(p^{-m})\beta( p^m).
\end{eqnarray*}
This does Lemma 2. 
\end{proof}
The rest of the terms on the right side of the prime power equation are evaluated in Lemma 3, Lemma 4, and Lemma 5. Continuing with the term, $- log \pi \ G(0)$, in the prime power equation,
\begin{lemma}
\begin{eqnarray*}
\sum_{j=1}^{\infty}\sum_{d|j}a[d] b[\frac{j}{d}]\Big(\frac{j}{d^2}\Big)^{\frac{1}{4}}\Big(- log \pi \ G(0, d, j)\Big) &=& - log \pi \ \sum_{j=1}^{\infty}\sum_{d|j}a[d] b[\frac{j}{d}] e^{-\pi( d+ \frac{j}{d})}\\
&=& -log\pi\sum_{k=1}^{\infty}a[k]exp[-\pi k] \sum_{l=1}^{\infty}b[l]exp[- \pi l] \\
&=& -log\pi \ \alpha(1) \beta(1).
\end{eqnarray*}
\end{lemma}
\begin{lemma}
\begin{eqnarray*}
\hat G(\frac{i}{2}) &=& \sum_{j=1}^{\infty}\sum_{d|j}a[d] b[\frac{j}{d}]\Big(\frac{j}{d^2}\Big)^{\frac{1}{4}} \hat G(\frac{i}{2}, 2\pi\sqrt{j}) \\
&=& 2\sum_{j=1}^{\infty}\sum_{d|j}a[d] b[\frac{j}{d}]K(0, 2\pi\sqrt{j}) \\
\hat G(-\frac{i}{2}) &=& \sum_{j=1}^{\infty}\sum_{d|j}a[d] b[\frac{j}{d}]\Big(\frac{j}{d^2}\Big)^{\frac{1}{4}} \hat G(-\frac{i}{2}, 2\pi\sqrt{j}) \\
&=& 2\sum_{j=1}^{\infty}\sum_{d|j}a[d] b[\frac{j}{d}]\Big(\frac{j}{d^2}\Big)^{\frac{1}{2}}   K(1, 2\pi\sqrt{j}).
\end{eqnarray*}
\end{lemma}
\begin{lemma}
\begin{eqnarray*}
\frac{1}{2\pi}\int_{-\infty}^{\infty}Re\Big(\frac{\dot \Gamma}{\Gamma}(\frac{1}{4} + \frac{i t}{2})  \Big)\hat G(t) dt &=& \int_0^{\infty}\frac{J(e^v)}{1 - e^{-2v}}dv -\gamma \alpha(1)\beta(1)\qquad \gamma = \mbox{Euler's Constant} \\
J(e^v) &=& 2e^{-2v}\alpha(1)\beta(1) -  e^{-v}\alpha(e^{-v})\beta(e^v) - \alpha(e^v)\beta(e^{-v}). 
\end{eqnarray*}
\end{lemma}
\begin{proof}
We follow Bombieri ([1], pp. 188-190). From
\begin{eqnarray*}
\Gamma[w] = \frac{1}{w}\prod_{n=1}^{\infty}\frac{(1+\frac{1}{n})^w}{1 + \frac{w}{n}},
\end{eqnarray*}
there is
\begin{eqnarray*}
\frac{\dot\Gamma}{\Gamma}(w)  &=& -\frac{1}{w} + \sum_{n=1}^{N}\Big(log(1 + \frac{1}{n})  - \frac{1}{(n + w)}\Big) +  \sum_{n > N}\Big(log(1 + \frac{1}{n})  - \frac{1}{(n + w)}\Big) \\
&=&  -\frac{1}{w} + log N - \sum_{n=1}^{N}\frac{1}{n + w} + \sum_{n > N}\Big (\frac{1}{n} - \frac{1}{n + w}\Big) + O\Big(\frac{1}{N}\Big) \\
 &=& log N - \sum_{n=0}^{N}\frac{1}{n + w} + O\Big(\frac{1 + |w|}{N}\Big) 
 \end{eqnarray*}
uniformly for $Re[w] \ge \frac{1}{4}$.
Then
\begin{eqnarray*}
Re\Big (\frac{\dot \Gamma}{\Gamma}(\frac{1}{4} + \frac{i t}{2}) \Big) = log N - \sum_{n=0}^{N} \frac{4n + 1}{(2 n + \frac{1}{2})^2 + t^2} +  O\Big(\frac{1 + |t|}{N}\Big).
\end{eqnarray*}
Thus, 
\begin{eqnarray}
\frac{1}{2\pi}\int_{-\infty}^{\infty}Re\Big(\frac{\dot \Gamma}{\Gamma}(\frac{1}{4} + \frac{i t}{2})  \Big)\hat G(t) dt &=& log N  \frac{1}{2\pi}\int_{-\infty}^{\infty}\hat G(t) dt \\
&-& \sum_{n=0}^{N} \frac{1}{2\pi}\int_{-\infty}^{\infty}\frac{4n + 1}{(2 n + \frac{1}{2})^2 + t^2}\hat G(t) dt \\
&+& O(\frac{1}{N}) \notag. 
\end{eqnarray}                                               
 As for (1), the factor, $log N$, is written as,
 \begin{eqnarray*}
 log N &=& \sum_{n=1}^N \frac{1}{n} - \gamma + O(\frac{1}{N}) \\
 &=& 2\int_0^{\infty}e^{-2v}\Big(\frac{1- e^{-2vN}}{1 - e^{-2v}}\Big)dv - \gamma + O(\frac{1}{N}). 
 \end{eqnarray*} 
 So (1) becomes,
 \begin{eqnarray}
 log N  \frac{1}{2\pi}\int_{-\infty}^{\infty}\hat G(t) dt &=& 2\int_0^{\infty}e^{-2v}\Big(\frac{1- e^{-2vN}}{1 - e^{-2v}}\Big) G(0) dv - \gamma G(0) + O(\frac{1}{N}) 
 \end{eqnarray}
As for (2), write
 \begin{eqnarray}
\frac{4n + 1}{(2 n + \frac{1}{2})^2 + t^2} = \frac{1}{i}\Big(\frac{1}{t - i (2n + \frac{1}{2})} - \frac{1}{t + i (2n + \frac{1}{2})}\Big),
 \end{eqnarray}
and note that,
 \begin{eqnarray}
\frac{1}{2\pi i}\int_{-\infty}^{\infty}e^{i t v}\Big[\frac{1}{t - i (2n + \frac{1}{2})} - \frac{1}{t + i (2n + \frac{1}{2})}\Big]d t = \text{Min}[e^{v (2n+1)}, e^{-v(2n+1)}].
 \end{eqnarray}  
Write (2) as a double integral, 
 \begin{eqnarray*}
 -\frac{1}{2\pi}\int_{-\infty}^{\infty}\frac{4n + 1}{(2 n + \frac{1}{2})^2 + t^2}\hat G(t)dt = -\frac{1}{2\pi}\int_{-\infty}^{\infty}\frac{4n + 1}{(2 n + \frac{1}{2})^2 + t^2}\int_{-\infty}^{\infty}e^{i t v}G(v) dv dt,
\end{eqnarray*}
substitute (4), interchange the order of integration, and use (5) to obtain for (2) the sums,
 \begin{eqnarray*}
- \sum_{n=0}^{N} \int_{-\infty}^{0}e^{v (2n+\frac{1}{2})}G(v) dv - \sum_{n=0}^{N} \int_{0}^{\infty}e^{-v (2n+\frac{1}{2})}G(v) dv, 
\end{eqnarray*} 
which equals, 
 \begin{eqnarray}
- \sum_{n=0}^{N} \int_{0}^{\infty}e^{v (2n+\frac{1}{2})}G(-v) dv - \sum_{n=0}^{N} \int_{0}^{\infty}e^{-v (2n+\frac{1}{2})}G(v) dv. 
 \end{eqnarray}
Combining (3) and (6), and letting  $N \longrightarrow \infty$, then,     
 \begin{eqnarray*}
\frac{1}{2\pi}\int_{-\infty}^{\infty}Re\Big(\frac{\dot \Gamma}{\Gamma}(\frac{1}{4} + \frac{i t}{2})  \Big)\hat G(t) dt = 
\end{eqnarray*}   
 \begin{eqnarray*}
\int_0^{\infty}\Big[2e^{-2v}G(0) - e^{-\frac{v}{2}}(G(-v)  + G(v) )\Big]\frac{dv}{1-e^{-2v}} 
-\gamma G(0).
\end{eqnarray*} 
The expression in the square brackets is
 \begin{eqnarray*}
 J(e^v) = 2e^{-2v}\alpha(1)\beta(1) -  e^{-v}\alpha(e^{-v})\beta(e^v) - \alpha(e^v)\beta(e^{-v}),
\end{eqnarray*} 
so that, as stated in the lemma,      
 \begin{eqnarray*}
\frac{1}{2\pi}\int_{-\infty}^{\infty}Re\Big(\frac{\dot \Gamma}{\Gamma}(\frac{1}{4} + \frac{i t}{2})  \Big)\hat G(t) dt = \int_0^{\infty}\frac{J(e^v)}{1 - e^{-2v}}dv -\gamma \alpha(1)\beta(1).
\end{eqnarray*}  
Changing variables, 
 \begin{eqnarray*} 
v = log u \qquad dv = \frac{du}{u},
\end{eqnarray*}
we have,
\begin{eqnarray*} 
\int_0^{\infty}\frac{J(e^v)}{1 - e^{-2v}}dv &=& \int_1^{\infty}\frac{2u^{-2}\alpha(1)\beta(1) -  u^{-1}\alpha(u^{-1})\beta(u) - \alpha(u)\beta(u^{-1})}{1 - u^{-2}}\frac{du}{u} \\
&=& \int_0^1\frac{2u^2\alpha(1)\beta(1) -  u\alpha(u)\beta(u^{-1}) - \alpha(u^{-1})\beta(u)}{1 - u^2}\frac{du}{u}  
\end{eqnarray*}
\end{proof}
The term, 
\begin{eqnarray*}
V(\infty) &=&log \pi \ G(0) \notag \\
 &+&\frac{1}{4\pi}\int_{-\infty}^{\infty}\Big(\frac{\dot \Gamma}{\Gamma}(\frac{1}{4} + \frac{i t}{2}) + \frac{\dot \Gamma}{\Gamma}(\frac{1}{4} - \frac{it}{2}) \Big)\hat G(t) dt,
\end{eqnarray*}
on the right side of the prime power equation comes from the Archimedian factor in $Z(s)$, a component of the functional equation on which the Guinand formula rests, and might as well be labelled that way. It plays the role of a known term in the equations, but a difficult expression, containing the values of $\beta(e^{-\pi y})$ in an integral. An effort was made to evaluate it as a discrete expression of special values of $\beta$.
\section {Introducing a Variable $x > 0$}
Continuous translates, $G(v + log x), x > 0,$ of $G$ are introduced,   from which a discrete infinite collection, $x = m, m = 1, 2,\cdots$ is selected to create an infinite set of prime power equations. For $x > 0$, it is readily verified that,
\begin{eqnarray*}
G(v + log x) = \sqrt{x}e^{\frac{v}{2}}\Big(\sum_{k=1}^{\infty} a(k)e^{-\pi k x e^v}\Big) \Big(\sum_{l=1}^{\infty} b(l)e^{-\frac{\pi l}{x} e^{-v}}\Big),
\end{eqnarray*} 
so that
\begin{eqnarray*}
G(m log p + log x) &=&  \sqrt{x} p^{\frac{m}{2}}\Big(\sum_{k=1}^{\infty} a(k)e^{-\pi k x p^m}\Big) \Big(\sum_{l=1}^{\infty} b(l)e^{-\frac{\pi l}{xp^m}}\Big) \\
&=&  \sqrt{x}p^{\frac{m}{2}}\alpha(x p^m)\beta(\frac{1}{xp^m}) \\
G(-m log p + log x) &=&  \sqrt{x}p^{-\frac{m}{2}}\Big(\sum_{k=1}^{\infty} a(k)e^{-\pi k x p^{-m}}\Big) \Big(\sum_{l=1}^{\infty} b(l)e^{-\frac{ \pi l p^m}{x}}\Big) \\
&=&  \sqrt{x}p^{-\frac{m}{2}}\alpha(\frac{x}{ p^m}) \beta(\frac{p^m}{x}). 
\end{eqnarray*}
With the variable, $x$, in all the terms, the modified prime power equation is now,
\begin{lemma}
\begin{eqnarray*}
\sqrt{x}\sum_{n=1}^{\infty}\Lambda (n) \  \alpha(x n)\beta(\frac{1}{x n}) +  \sqrt{x}\sum_{n=1}^{\infty}\frac{\Lambda(n)}{n} \alpha(\frac{x}{n})\beta(\frac{n}{x}) 
&=& 
\frac{1}{\sqrt{x}} \hat G(\frac{i}{2}) \\
&+&
\sqrt{x} \hat G(-\frac{i}{2}) \\
&-&
\sqrt{x} \ log \pi \ \alpha(x) \beta(\frac{1}{x}) \\
&+& \sqrt{x}\int_{0}^{\infty}\frac{J(e^v, x)}{1 - e^{-2v}}dv -\gamma \sqrt{x}\  \alpha(x)\beta (\frac{1}{x}),
\end{eqnarray*}
\end{lemma}
where
\begin{eqnarray*}
J(e^v, x) &=& 2 e^{-2v}\alpha(x)\beta(\frac{1}{x}) - e^{-v}\alpha(xe^{-v})\beta(\frac{e^v}{x}) - \alpha(xe^{v})\beta(\frac{e^{-v}}{x}).
\end{eqnarray*}
In order to remove $\hat G(\frac{i}{2})$ and $\hat G(-\frac{i}{2})$ from the prime power equation, note that, given a relation,
\begin{eqnarray*}
f(x) &=& \sqrt{x} g(x) \\
&=& a \sqrt{x} + \frac{b}{\sqrt{x}} + h(x), 
\end{eqnarray*}
as exists in our prime power formula, then the operation,
\begin{equation*}
f(x) + f(\frac{1}{x}) - f(1)(\sqrt{x} + \frac{1}{\sqrt{x}})
\end{equation*}
eliminates
\begin{equation*}
a \sqrt{x} + \frac{b}{\sqrt{x}}.
\end{equation*}
This operation removes $\hat G[\frac{i}{2}]$ and $\hat G[-\frac{i}{2}]$. The resulting prime power formula would depend only on $\alpha$ and $\beta$, but as the situation stands, there is no evident advantage in taking this step, and the right side, $V(G)$, is left with the two Bessel terms.
\section {Solving for Prime Powers}             
\noindent Writing the prime power equation as,
\begin{equation*}
\sum_{n=1}^{\infty}  f(x, n) \Lambda(n)= V(x)\qquad x  > 0,
\end{equation*}
and replacing $x$ with $m = 1, 2, \cdots$, in turn, there results a sequence of equations,
\begin{equation*}
\sum_{n=1}^{\infty}  f(m, n)\Lambda(n) = V(m), 
\end{equation*}
indexed by the natural numbers, $[m = 1, 2, \cdots]$. With the pair of column vectors,
\begin{equation*}
\Lambda = [\Lambda(n)]_{n=1}^{\infty},
\end{equation*}
and,
\begin{equation*}
V = [V(m)]_{m =1}^{\infty}, 
\end{equation*}
the linear system of equations is put in the form,
\begin{equation*}
T \Lambda  = V.
\end{equation*}
The entries of 
\begin{equation*}
T = [f (m, n)]_{m, n = 1}^{\infty},
\end{equation*}
are
\begin{equation*}
f(m, n) = \alpha(m n)\beta(\frac{1}{m n}) + \frac{1}{n} \alpha(\frac{m}{n})\beta(\frac{n}{m}),
\end{equation*}
where the square root, $\sqrt{m}$, was put into the $V$ vector. The objective is a solution,
\begin{equation*}
\Lambda  = T^{-1}V.
\end{equation*}
From any angle this does not look promising. 
\section {The Invertibility of $T$}
Firstly, the possible inverse, $T^{-1}$, cannot be approached as a limit of inverses of upper left hand $n\times n$ cutouts of $T$. These are not approximations to T. The diagonal entries of $T$ are
\begin{equation*}
f(n, n) = \frac{\alpha(1)\beta(1)}{n} + \epsilon(n) \qquad \epsilon(n) \longrightarrow 0.
\end{equation*}  
Since $\alpha$ and $\beta$ are positive,
\begin{eqnarray*}
\sum_{k=n+1}^{2n-1}f(n, k) &\ge& \sum_{k=n+1}^{2n}\frac{1}{k}\alpha\Big(\frac{n}{k+1}\Big)\beta\Big(\frac{k+1}{n}\Big) \\
&\ge& log 2 \ \alpha(1)\beta(2). 
\end{eqnarray*}
Thus $T$ is not diagonally dominant, in fact, as shown, far from it. And the arguments used to establish the existence of an inverse for infinite systems, which rely on finite dimensional (upper left) approximations to $T$, cannot be applied. \\
Note that
\begin{eqnarray*}
\sum_{k=n+1}^{2n-1}f(n, k) \Lambda(k) &\ge& \sum_{k=n+1}^{2n}\frac{\Lambda(k)}{k}\alpha\Big(\frac{n}{k+1}\Big)\beta\Big(\frac{k+1}{n}\Big) \\
&\ge& log 2 \ \alpha(1)\beta(2), 
\end{eqnarray*} 
so the action of the $n$th row of the finite dimensional, $n\times n$, upper left truncation, $T_n$, of $T$, on $\Lambda$ differs from the action of the $n$th row of $T$ on $\Lambda$ by a constant for all $n$.
\section {The Matrix T} 
Secondly, assuming an inverse, $T^{-1}$, finding it contains the known complications of finding an inverse in the finite dimensional case. Even matrices with simple entries might have an inverse that relies on a highly sophisticated combinatorial identity. Basically, they're  works of art. Consider what there is to encounter in the entries of $T$. Since 
\begin{eqnarray*}
\alpha(u) = \sum_{l=0}^{\infty}2^l \beta(2^l u), 
\end{eqnarray*}
the matrix, $T$, depends on $\beta$ alone, and consists of sums and products of values, 
\begin{equation*}
\beta(e^{-\pi r})\qquad r\in Q^{+},
\end{equation*}
which are, to within a factor of $\frac{1}{16}$,  equal to 
\begin{equation*}
\theta_2^4(\tau )\qquad q = e^{i\pi\tau}, \tau = i r. 
\end{equation*}
For $r = N$, a positive integer, 
\begin{equation*}
\theta_2^4(N \tau) = \Big(\frac{\theta_3(N\tau)}{\theta_3(\tau)}\Big)^4 \lambda(N\tau)\theta_3^4(\tau),
\end{equation*}
and the factor,
\begin{equation*}
 F(N \tau) = \Big(\frac{\theta_3(N\tau)}{\theta_3(\tau)}\Big)^4 \lambda(N\tau),
\end{equation*}
is algebraic at $\tau = i$. With $N = p$, a prime, this follows from three results. First, that
\begin{equation*}
\frac{\theta_3(p \tau)}{\theta_3(\tau)} 
\end{equation*}
is an algebraic function of $\lambda(\tau)$; second,  that $\lambda(p \tau)$ is an algebraic function of $\lambda(\tau)$; and, third, that $\lambda(i) = \frac{1}{2}$. The values of $F(N i)$ are then algebraic functions of $\lambda(i)$, being compositions indexed on the prime factors of $N$. For  rational $r = \frac{N}{M}> 0$, $F(r i)$ is algebraic, arguing as before with $\tau = \frac{i}{M}$, using the relation,
\begin{equation*}
\lambda(\frac{i}{M}) = 1 - \lambda(i M).
\end{equation*}
Finally,
\begin{eqnarray*}
\theta_3(i) &=& \frac{\pi^{\frac{1}{4}}}{\Gamma(\frac{3}{4})}.
\end{eqnarray*} 
Thus, the values, $\theta_2^4( r i)$, the building blocks of the entries of $T$, are algebraic multiples of $\theta_3^4(i)$.
\section {Supposing $T^{-1}$} 
Write,
\begin{equation*}
\Lambda = T^{-1}V.
\end{equation*}
Let
\begin{equation*}
U = [u_{mn}]_{m,n=1}^{\infty}
\end{equation*}
\begin{eqnarray*}
u_{mn} &=& 1\qquad n \leq m \\
&=& 0 \qquad n > m \\
&=& \frac{1}{2} \qquad n =m = p^k,
\end{eqnarray*}
then
\begin{equation*}
U \Lambda = \Psi_0,
\end{equation*}
where $\Psi_0$ is the column vector,
\begin{equation*}
[\Psi_0(N)]_{N=1}^{\infty},
\end{equation*}
with                    
\begin{equation*}
\Psi_0(N) = \sum_{n\le N}u_{Nn}\Lambda(n).
\end{equation*}
Having reached this point, it would then seem that technique is within reach to reconcile this formula with the explicit,
\begin{equation*}
\Psi_0(N) = N - \sum_{\rho}\frac{N^{\rho}}{\rho} -  \frac{\dot\zeta}{\zeta}(0) - \frac{1}{2} log(1 - N^{-2}) \qquad N > 1,
\end{equation*}
and identify the zeros of the zeta function originating in the structure of $T$.  The point here is that a mystery has simply been recast in a slightly novel setting. None of the mystery has been eliminated. The unknown, $T^{-1}$, was therefore an expected stopping point, though a $T^{-1}$ might be investigated in the context of function fields, where the (norms of) primes have an internal structure,
\begin{equation*}
p = q^{e(p)}\qquad e(p) \in Z^{+}.
\end{equation*}
With
\begin{equation*}
e_n  = \Big| [p: e(p) = n] \Big |,
\end{equation*}
the sequence,
\begin{equation*}
N_k = \sum_{d | k}d e_d
\end{equation*} 
satisfies a finite linear recursion over the integers. Its characteristic polynomial produces the zeros of a zeta function. If 
\begin{eqnarray*}
p &=& e^{log p} \notag \\
 &=& \prod_{n=1}^{\infty}e^{T_{mn}^{-1}V(n)}\qquad m = m(p),
\end{eqnarray*}
is the internal structure in rational primes, then a connection over $m(p)$ might be discoverable. It could replace the finite recursion, and the zeta zeros would come out of that connection.
\section {Other $G$, $\hat G$}
As we saw in Lemma 1, a $G$ can be constructed from the modular form, $\theta_4^4$, which provides the equation, 
\begin{eqnarray*}
\frac{s(s+1)}{32 \pi^2\sqrt{2}}(2^{\frac{s}{2}} - 2^{-\frac{s}{2}})(2^{\frac{s-1}{2}} - 2^{-\frac{s-1}{2}})\zeta^{*}(s)\zeta^{*}(s+1) = \sum_{j=1}^{\infty}\sum_{d|j}a(d)b(\frac{j}{d})\Big(\frac{j}{d^2}\Big)^{\frac{s}{2}}K[s, 2\pi\sqrt{j}].
\end{eqnarray*}
As the translates of $G$ created equations with the question of an inverse unanswered (more general convolutions were no more helpful), the search turned to $G$ by other constructions. One such $G$ comes from $\theta_4^8$, which, by the same procedure as with $\theta_4^4$, gives 
\begin{eqnarray*}
\frac{s(s+2)(s + 3)}{256 \pi^3\sqrt{2}}2^{\frac{s}{2}}(2^{\frac{s - 1}{2}} - 2^{-\frac{s - 1}{2}})\zeta^{*}(s)\zeta^{*}(s+3) = \sum_{j=1}^{\infty}\sum_{d|j}a(d)B(\frac{j}{d})\Big(\frac{j}{d^2}\Big)^{\frac{s}{2}}K[s, 2\pi\sqrt{j}].
\end{eqnarray*}
This paper could have been written with a $G$ and $\hat G$ coming from $\theta_4^8$. Then  $\beta^2$ would replace $\beta$. Everything else, including $\alpha$, would be the same. The only powers of $\theta_4$ that work in this way, however, are $1, 2, 4, 8$. A countable replacement for the $x$ variable that would provide an infinite set of equations was sought from the imaginary quadratic extensions. There are infinitely many of these, and the $\theta_3$ type modular forms, (two-dimensional) on which the zeta functions are constructed, might be replaced with $\theta_4$ type modular forms that support the steps that give,
\begin{eqnarray*}
E_{L}(s)\zeta_L(s) = \hat G_L(\frac{s - \frac{1}{2}}{i}).
\end{eqnarray*}
For example, the form producing $\zeta_L(s)$ for
\begin{eqnarray*}
L = Q(\sqrt{-1})
\end{eqnarray*}
is
\begin{eqnarray*}
\sum_{m,n=-\infty}^{\infty}e^{i\pi\tau (m^2 + n^2)},
\end{eqnarray*}
but
\begin{eqnarray*}
\sum_{m,n=-\infty}^{\infty}(-1)^m (-1)^n e^{i\pi\tau (m^2 + n^2)},
\end{eqnarray*}
also gives a $\zeta_L(s)$, with a factor $1 - 2^{1-s}$.
The field,
\begin{eqnarray*}
L = Q(\sqrt{-2})
\end{eqnarray*}
with the form,
\begin{eqnarray*}
\sum_{m,n=-\infty}^{\infty}e^{i\pi\tau (m^2 + 2n^2)},
\end{eqnarray*}
is the other field with $1 - 2^{1-s}$. The other seven class number one fields have a factor,
\begin{eqnarray*}
1 - 4^{1-s},
\end{eqnarray*}
as do  $Q(\sqrt{-5})$ and $Q(\sqrt{-6})$. These are a bit tricky, and whether the good construction persists is unknown. Even if there are infinitely many good quadratic fields, it's another $T$ dragging unsolved problems from one branch of number theory into the question of another $T^{-1}$. The opposite consideration is there, though, that the structure within which the prime power equation should be considered, if it should be considered at all, is necessarily a great deal more polished than the structure within which it appeared. \\
\section {Bibliography} 
1.  E. Bombieri, Remarks on Weil's quadratic functional in the theory of prime numbers, I.  Rend. Mat. Acc. Lincei s. 9, v. 11:183-233 (2000)\\
2. H. Yoshida, On Hermitian Forms attached to Zeta Functions. In: N. Kurokawa - T. Sunada (eds.), Zeta Functions in Geometry. Advanced Studies in Pure Mathematics, 21, Mathematical Society of Japan, Kinokuniya, Tokyo 1992, 281 - 325.
\end{document}